\documentclass[11pt,oneside,english]{amsart}
\usepackage[T1]{fontenc}
\usepackage[latin9]{inputenc}
\usepackage{babel}
\usepackage{verbatim}
\usepackage{amsthm}
\usepackage{amstext}
\usepackage{amssymb}
\usepackage[all]{xy}
\usepackage[unicode=true,pdfusetitle,
 bookmarks=true,bookmarksnumbered=false,bookmarksopen=false,
 breaklinks=false,pdfborder={0 0 1},backref=false,colorlinks=false]
 {hyperref}
\usepackage{breakurl}

\makeatletter
\numberwithin{equation}{section}
\numberwithin{figure}{section}
  \theoremstyle{plain}
  \newtheorem*{thm*}{\protect\theoremname}
\theoremstyle{plain}
\newtheorem{thm}{\protect\theoremname}
  \theoremstyle{plain}
  \newtheorem{conjecture}[thm]{\protect\conjecturename}
  \theoremstyle{definition}
  \newtheorem{defn}[thm]{\protect\definitionname}
  \theoremstyle{plain}
  \newtheorem{lem}[thm]{\protect\lemmaname}
  \theoremstyle{remark}
  \newtheorem{rem}[thm]{\protect\remarkname}
  \theoremstyle{plain}
  \newtheorem{prop}[thm]{\protect\propositionname}
  \theoremstyle{plain}
  \newtheorem{cor}[thm]{\protect\corollaryname}

\newcommand{\ra}{\rightarrow}

\newcommand{\lra}{\longrightarrow}

\newcommand{\sra}[1]{\stackrel{#1}{\ra}}

\newcommand{\ov}{\overline}

\newcommand{\cA}{{\mathcal A}}
\newcommand{\cB}{{\mathcal B}}

\newcommand{\cC}{{\mathcal C}}

\newcommand{\cE}{{\mathcal E}}

\newcommand{\cN}{{\mathcal N}}
\newcommand{\cO}{{\mathcal O}}
\newcommand{\cP}{{\mathcal P}}

\newcommand{\cX}{{\mathcal X}}
\newcommand{\cY}{{\mathcal Y}}

\newcommand{\bN}{{\mathbb N}}

\newcommand{\bR}{{\mathbb R}}
\newcommand{\bC}{{\mathbb C}}
\newcommand{\bZ}{{\mathbb Z}}
\newcommand{\bQ}{{\mathbb Q}}

\newcommand{\bA}{{\mathbb A}}

\newcommand{\fg}{{\mathfrak g}}

\makeatother

  \providecommand{\conjecturename}{Conjecture}
  \providecommand{\corollaryname}{Corollary}
  \providecommand{\definitionname}{Definition}
  \providecommand{\lemmaname}{Lemma}
  \providecommand{\propositionname}{Proposition}
  \providecommand{\remarkname}{Remark}
  \providecommand{\theoremname}{Theorem}
\providecommand{\theoremname}{Theorem}

\begin{document}

\title{Arithmetic groups with isomorphic finite quotients}

\author{Menny Aka}
\begin{abstract}
Two finitely generated groups have the same set of finite quotients
if and only if their profinite completions are isomorphic. Consider
the map which sends (the isomorphism class of) an S-arithmetic group
to (the isomorphism class of) its profinite completion. We show that
for a wide class of S-arithmetic groups, this map is finite to one,
while the the fibers are of unbounded size. 
\end{abstract}
\maketitle

\section{Introduction}

A classical topic in number theory is the study of integral quadratic
forms of the same genus, i.e., those who become isomorphic integrally
in every local completion. In the 70's Grunewald, Pickel and Segal
(cf. \cite{GPS80} and the references therein) studied non-commutative
analogues, establishing:
\begin{thm*}
\cite{GPS80}\label{thm:GPS}Let $\cP$ be the family of polycyclic-by-finite
groups and for $\Gamma\in\cP$ let $\cP_{\Gamma}:=\{\Lambda\in\cP:\widehat{\Lambda}\cong\widehat{\Gamma}\}$
where $\widehat{\Gamma}$ denotes the profinite completion of $\Gamma$.
Then, for any $\Gamma\in\cP$, $\cP_{\Gamma}$ is a finite union of
isomorphism classes.
\end{thm*}
Note that for finitely generated groups, having isomorphic profinite
completion is equivalent to having the same set of finite quotients
(see subsection \ref{sub:Profinite-groups}). Thus, in other words,
the Theorem says that in the family $\cP$ an element is determined
up to finitely many options by the set of its finite quotients. 

Recently, Grunewald and Zalesskii \cite{GZ11} revisited the topic.
We will follow their terminology: Let $\cC$ be a family of groups
and $\Gamma\in\cC$. The $\cC$- genus of $\Gamma$ is 
\[
g(\cC,\Gamma):={\rm IsoClasses}(\{\Lambda\in\cC:\widehat{\Lambda}\cong\widehat{\Gamma}\})
\]
which is the set of isomorphism classes of groups from $\cC$ whose
profinite completion is isomorphic to $\widehat{\Gamma}$. In this
paper we study the genus of $S$-arithmetic subgroups of simple algebraic
groups. Let us recall (following \cite[4.1]{PR94}) their definition:

Let $\cA$ be the family of groups $\Gamma$ such that: 
\begin{enumerate}
\item There exists $n\in\bN$, a number field $k$ and a $k$-algebraic
subgroup $G\subset GL_{n}$ such that $G$ is simply-connected, almost
and absolutely simple.\label{group definition} 
\item The group $\Gamma$ is isomorphic to a subgroup of $G(\bar{k})$ and
commensurable to $G(\cO_{k,S}):=G\cap GL_{n}(\cO_{k,S})$ where $S$
is a finite set of places of $k$ containing the Archimedean ones
and $\cO_{k,s}$ is the ring of $S$-integers (defined below). 
\item The $S$-rank of $G$ is $\geq2$, i.e., $\sum_{v\in S}rank_{k_{v}}(G)\geq2$
where $rank_{k_{v}}(G)$ is the dimension of a maximal $k_{v}$-split
tori in $G(k_{v})$ .
\end{enumerate}
Recall that two groups $\Gamma$ and $\Lambda$ are called \emph{commensurable}
if they have finite index subgroups $\Gamma_{1}<\Gamma,\Lambda_{1}<\Lambda$
with $\Gamma_{1}\cong\Lambda_{1}$ (see also Definitions \ref{def: commen}
and \ref{def: family A concrete} below).

The group structure of elements of $\cA$ encodes, in some sense,
the arithmetic information that is used to define them. Many rigidity
results about members of $\cA$ have been proved, in particular Margulis'
super-rigidity \cite{Mar91} implies that whenever two elements of
$\cA$ are isomorphic, there is a unique isomorphism between their
ambient groups over a unique isomorphism of their fields of definition.
In contrast to Margulis super rigidity, we find in Section \ref{sec:Examples}
and in \cite{Aka10} pairs of groups $\Gamma,\Lambda\in\cA$ that
have isomorphic profinite completion but are not isomorphic. Moreover,
the associated algebraic and arithmetic information that defines $\Gamma$
and $\Lambda$ can be quite different. Nevertheless, for $\Gamma\in\cA$
let 
\[
\cA_{\Gamma}:=\{\Lambda\in\cA:\widehat{\Lambda}\cong\widehat{\Gamma}\},
\]
and we conjecture the following:
\begin{conjecture}
\label{con:finiteness}For all $\Gamma\in\cA$, the set $\cA_{\Gamma}$
is a disjoint union of finitely many isomorphism classes.%
{} 
\end{conjecture}
In this paper we prove:
\begin{thm}
\label{thm:CSP}Assume that $\Gamma\in\cA$ has the congruence subgroup
property (see \ref{sub:CSP}) then the set $\cA_{\Gamma}$ is a disjoint
union of finitely many isomorphism classes.%
{} 
\end{thm}
Note that Serre's conjecture \cite[9.5]{PR94} states that all elements
of $\cA$ posses the congruence subgroup property and indeed this
conjecture has been proved for most $\Gamma$'s in $\cA$ (see \cite{PR08}
for the current state of this conjecture). Without assuming the congruence
subgroup property we show:
\begin{thm}
\label{Main Theorem 1}For all $\Gamma\in\cA$, the set $\cA_{\Gamma}$
is contained in the union of finitely many commensurability classes.
\end{thm}
The profinite completions of $\Gamma\in\cA$ is strongly related to
the $p$-adic closure of $\Gamma$. Thus, in some sense, the proofs
of the above results relies on group and representation theoretic
methods that are used for extracting this {}``local'' information
about $\Gamma$, and then using some local-to-global finiteness results.

Theorem \ref{thm:CSP} answers Problem I of \cite{GZ11} for the family
of $S$-arithmetic groups, and the examples we give in $\S4$ answer
problems II and III for the family of $S$-arithmetic groups. Moreover,
the proofs of Theorems \ref{thm:CSP},\ref{Main Theorem 1} can give
bounds on the size of $g(\cA,\Gamma)$ for various $\Gamma\in\cA$.
In particular we show,
\begin{thm}
\label{thm:Unbounded genus}For any $n\in N$ there exists $\Gamma\in\cA$
with $|g(\cA,\Gamma)|>n$.
\end{thm}
This paper is organized as follows: After giving the necessary definitions
and preliminaries in $\S2$, we prove Theorem \ref{Main Theorem 1}
in $\S3$. In $\S4$ we give examples of non-isomorphic elements of
$\cA$ that share the same profinite completion. These examples show
that our proof of Theorem \ref{Main Theorem 1} is {}``tight'' in
the sense that all the steps of the proof are reflected in non-trivial
examples. Moreover, these examples provide three different types of
families which establish Theorem \ref{thm:Unbounded genus}. In $\S5$
we restrict ourselves to a fixed commensurability class $\cC$ whose
elements admit the Congruence Subgroup Property and show that $\cC_{\Gamma}=\{\Lambda\in\cC:\widehat{\Gamma}\cong\widehat{\Lambda\}}$
is a disjoint union of finitely many isomorphism classes. Finally,
we deduce Theorem \ref{thm:CSP} in section \ref{sec:deduce main thm}.

\subsection*{Acknowledgments}

This paper is a part of the Author's Ph.D. thesis at the Hebrew University.
I would like to thank my adviser for suggesting the above subject
and for his guidance and ideas. I discussed various parts of this
paper with many people. In particular, I would like to thank Nir Avni,
Chen Meiri and Andrei Rapinchuk. During the period of work on this
paper, I enjoyed the support of the ERC, ISEF foundation, Ilan and
Asaf Ramon memorial foundation and the \textquotedbl{}Hoffman Leadership
and Responsibility\textquotedbl{} fellowship program, at the Hebrew
University.

\section{Notation, conventions and preliminaries}

We work with the definitions and notation from \cite{PR94}; relevant
results can be found there (especially in chapters 1 and 5).

\subsection{Field and group-theoretic notation}

For a number field $k$ we let $\cO_{k}$ be the ring of integers
of $k$, and $V^{k}$,$V_{f}^{k}$,$V_{\infty}^{k}$ - the set of
all places of $k$, non-Archimedean places of $k$, Archimedean places
of $k$, respectively. For $v\in V^{k}$, we let $k_{v}$ be the corresponding
completion, and $\cO_{v}$ is the ring of integers of $k_{v}$. The
letter $S$ always denotes a finite subset of $V^{k}$ which contains
$V_{\infty}^{k}$.%
{} Let $\cO_{k,S}:=\{x\in\cO_{k}:v(x)\geq0,\forall v\notin S\}$ be
the ring of $S$-integers. Finally, $\bA_{k}$ denoted the Adeles
ring of $k$ and $\bA_{k,S}$ the ring of $S$-Adeles (See section
\ref{sub:Adeles} and \cite[1.25]{PR94}). When $\Gamma_{1}$ is of
finite-index in $\Gamma_{2}$, we write $\Gamma_{1}<_{fi}\Gamma_{2}$.

\subsection{A {}``working'' definition for the family $\cA$.\label{sub:working-definition A}}

In order to make the reduction steps of the proof of Theorem \ref{thm:CSP}
precise, we will need a more technical definition of the family $\cA$.
We begin by defining the notion of commensurability:
\begin{defn}
\label{def: commen}Two subgroups $\Gamma_{1},\Gamma_{2}$ of a group
$H$ are said to be \emph{commensurable} if $\Gamma_{1}\cap\Gamma_{2}$
is of finite-index in both $\Gamma_{1}$ and $\Gamma_{2}$. Two groups
$\Gamma_{1},\Gamma_{2}$ are said to be \emph{abstractly commensurable}
if there is a group $H$ containing isomorphic copies $\Gamma_{1}',\Gamma_{2}'$
of $\Gamma_{1},\Gamma_{2}$ respectively, such that $\Gamma_{1}'$
and $\Gamma_{2}'$ are commensurable.
\end{defn}
Since isomorphic $\Gamma,\Gamma'\in\cA$ may have different ambient
information, it will be convenient to us to record the ambient information
as part of the definition of $\cA$:
\begin{defn}
\label{def: family A concrete} An element of $\cA$ is a quadruple
$(\Gamma,(G,\rho),k,S)$ such that 
\begin{enumerate}
\item $k$ is a number field, i.e., a finite extension of $\bQ$, and $G\sra{\rho}GL_{n}$
is a faithful $k$-representation of $G$ as a matrix group, in which
$G$ is a simply-connected, almost and absolutely simple and defined
over $k$. We identify $G$ with its image under $\rho$.
\item $ $The group $\Gamma\subset G(\bar{k})$ is commensurable to $G(\cO_{k,S}):=G\cap GL_{n}(\cO_{k,S})$
where $S$ is a finite subset of $V^{k}$ containing $V_{\infty}^{k}$. 
\item The $S$-rank of $G$ is $\geq2$, i.e., $\sum_{v\in S}Rank_{k_{v}}(G)\geq2$
where $Rank_{k_{v}}(G)$ is the dimension of a maximal $k_{v}$-split
tori in $G(k_{v})$ .
\end{enumerate}
We call $(G,k,S)$ the ambient information of $\Gamma$, $G$ the
ambient group of $\Gamma,$ $k$ the field of definition of $\Gamma$
and $S$ the valuations of $\Gamma$, and denote them by $G_{\Gamma},k_{\Gamma},S_{\Gamma}$
respectively. Often, we will leave them implicit. The \emph{commensurability
class} of an element $\Gamma\in\cA$ is the set of all elements of
$\cA$ which are abstractly commensurable to it. 
\end{defn}
We wish to make some further assumption on $S_{\Gamma}$ without restricting
the generality of the family $\cA$.
\begin{lem}
\label{lem:finiteindex free of compacts}Let $\Gamma=(\Gamma,G,k,S)\in\cA$
and let $C:=\{v\in V_{f}^{k}:G(k_{v})\text{ is compact}\}.$ Then
$G(\cO_{k,S\setminus C})$ is a finite-index subgroup of $G(\cO_{k,S})$.
In particular, $\Gamma=(\Gamma,G,k,S\setminus C)\in\cA$ can also
serve as ambient information for $\Gamma$.\end{lem}
\begin{proof}
Recall that $\cO_{k,T}:=\{x\in k:x\in\cO_{v},\forall v\notin T\}$.
Using this, one can see that the diagonal embedding $G(\cO_{k,S})\ra\prod_{v\in S}G(k_{v})$
induces the inequality
\begin{align*}
[G(\cO_{k,S}):G(\cO_{k,S\setminus C})] & \leq[\prod_{v\in S}G(k_{v}):\prod_{v\in S\setminus C}G(k_{v})\times\prod_{v\in C}G(\cO_{v})]\\
 & =\prod_{v\in C}[G(k_{v}):G(\cO_{v})].
\end{align*}
Each factor of the latter is the number of cosets of an open subgroup
of a compact group, hence finite. 
\end{proof}
From now on, whenever an element $\Gamma=(\Gamma,G,k,S)\in\cA$ is
specified, we will always assume that $S$ does not contain a valuation
$v$ for which $G(k_{v})$ is compact. 
\begin{lem}
\label{lem:k-iso commen}Let $(\Gamma,G,k,S)\in\cA$ and let $\varphi:G\ra G'$
a homomorphism of algebraic groups where $G'$ is also defined over
$k$ and $\varphi$ is a $k$-isomorphism. Then, $\varphi(\Gamma)$
is commensurable to $G'(\cO_{k,S})$.\end{lem}
\begin{proof}
See \cite[Theorem 4.1]{PR94}.
\end{proof}
We will have a particular interest in the rational primes that has
no valuation in $S$ above them; given a finite subset $S\subset V^{k}$
we denote by $S^{full}$ the set of rational primes $p$ for which
${v|p}$ imply that $v\notin S$. We call a prime number $p$\emph{
full for} $\Gamma$ if $p\in S_{\Gamma}^{full}:=(S_{\Gamma})^{full}$. 

Note that we do not impose the condition that $\Gamma$ lie in $G_{\Gamma}(k_{\Gamma})$.
If it does lie in it, i.e. if $\Gamma\subset G_{\Gamma}(k_{\Gamma})$,
we say that $\Gamma$ is \emph{rational}. Note that in many cases,
maximal arithmetic subgroups lie outside $G_{\Gamma}(k_{\Gamma})$.

\subsection{Profinite completion of a group\label{sub:Profinite-groups}}

For background on profinite groups the reader can consult the book
of Ribes and Zalesskii \cite{RZ2010}. For completeness, we record
here the definition of the profinite completion of a group. Let $\Gamma$
be a finitely generated group. The profinite topology on $\Gamma$
is defined by taking as a fundamental system of neighborhoods of the
identity the collection of all normal subgroups $N$ of $\Gamma$
such that $\Gamma/N$ is finite. One can easily show that a subgroup
$H$ is open in $\Gamma$ if and only if $H<_{fi}\Gamma$. We can
complete $\Gamma$ with respect to this topology to get 
\[
\widehat{\Gamma}:=\varprojlim\{\Gamma/N:N\vartriangleleft_{fi}\Gamma\};
\]
this is the profinite completion of $\Gamma$. There is a natural
homomorphism $i:\Gamma\rightarrow\widehat{\Gamma}$ given by $i(\gamma)=\varprojlim(\gamma N)$.
A group is called \emph{residually finite} if $i$ is injective. One
can show (\cite[Corollary 3.2.8]{RZ2010}) that for finitely generated
groups $\widehat{\Gamma_{1}}\cong\widehat{\Gamma_{2}}$ if and only
if their set of finite quotients are equal.
\begin{lem}
\label{lem:subgroups of profinite completion} Let $\Gamma$ be a
residually finite group. There is a one-to-one correspondence between
the set $\cX$ of all finite-index subgroups of $\Gamma$ and the
set $\cY$ of all open subgroups of $\widehat{\Gamma}$, given by
\begin{gather*}
X\mapsto cl(X)\quad(X\in\cX)\\
Y\mapsto Y\cap\Gamma\quad(Y\in\cY)
\end{gather*}
where $cl(X)$ denotes the closure of $X$ in $\widehat{\Gamma}$.
Moreover, 
\[
[\Gamma:X]=[\widehat{\Gamma}:cl(X)].
\]
\end{lem}
\begin{proof}
See \cite[Window: profinite groups, Proposition 16.4.3]{LubotzkySegal} 
\end{proof}
As members of $\cA$ are residually finite we can use Lemma \ref{lem:subgroups of profinite completion},
to make the following definition: 
\begin{defn}
\label{def: FI corres} Given an isomorphism $\Phi:\widehat{\Gamma}\ra\widehat{\Lambda}$
we say that $\Lambda^{0}<_{fi}\Lambda$ correspond to $\Gamma^{0}<_{fi}\Gamma$
via $\Phi$ if $\Lambda^{0}=\Phi(\widehat{\Gamma^{0}})\cap\Lambda$
and denote $\Phi_{*}(\Gamma^{0})=\Lambda^{0}$.
\end{defn}
This sets up a correspondence between finite-index subgroups of $\Gamma$
and finite-index subgroups of $\Lambda$ which respects the relation
of containment. Moreover, note that by Lemma \ref{lem:subgroups of profinite completion}
we have $[\Lambda:\Phi_{*}(\Gamma^{0})]=[\Gamma:\Gamma^{0}]$.

\subsection{Adeles groups\label{sub:Adeles}}

For an algebraic group $G\subset GL_{n}$ that satisfies the assumptions
in Definition \ref{def: family A concrete}, let 
\[
G(\bA_{k}):=\underset{v\in V^{k}}{{\prod}^{*}}G(k_{v}),\quad G(\bA_{k,S}):=\underset{v\in V^{k},v\notin S}{{\prod}^{*}}G(k_{v})
\]
were $*$ denotes the restricted product of $G(k_{v})$ for all $v\in V^{k}$
with respect to $G(\cO_{v}):=G\cap GL_{n}(\cO_{v})$ (see \cite[\S 5.1]{PR94}
for details). For $v\notin S$, we denote by $\pi_{v}:G(\bA_{k,S})\ra G(k_{v})$
and $\pi_{p}:G(\bA_{k,S})\ra\prod_{v\notin S,v|p}G(k_{v})$ the natural
projections and we denote by $i_{v}:G(k_{v})\ra G(\bA_{k,S})$ and
by $i_{p}:\prod_{v\notin S,v|p}G(k_{v})\ra G(\bA_{k,S})$ the natural
injections.

\subsection{The congruence subgroup property\label{sub:CSP}}

We give here a brief survey; for proofs of the assertions below the
reader can consult \cite{PR08}. 

Given an element of $\cA$ of the form $G(\cO_{k,S})$ (as always,
there is an implicit embedding of $G$ to $GL_{n}$), for any non-zero
ideal $I$ of $\cO_{k,S}$ let $K_{I}$ be the kernel of the map 
\begin{equation}
G(\cO_{k,S})\ra G(\cO_{k,S}/I).\label{eq:Cong map}
\end{equation}
The completion of $G(\cO_{k,S})$, with respect to the topology in
which these kernels form a system of neighborhoods of the identity,
is called the congruence completion. We denote this completion by
$\ov{G(\cO_{k,S})}$. By strong approximation for $G$ w.r.t. S, the
maps in \ref{eq:Cong map} are surjective for all but finitely many
ideals $I$ (See \cite[7.4]{PR94}). Using this, one can show that
$\overline{G(\cO_{k,S})}$ is naturally isomorphic to the open compact
subgroup $\prod_{v\notin S}G(\cO_{v})$ of $G(\bA_{k,S})$. We have
following the short exact sequence 
\[
1\ra C(S,G)\ra\widehat{G(\cO_{k,S})}\ra\ov{G(\cO_{k,S})}\ra1
\]
and $C(S,G)$ is called the congruence kernel of $G$ w.r.t. $S$.
If $C(S,G)$ is finite, we say that $G$ admits the congruence subgroup
property. Finally, if $\Gamma$ is commensurable to $G(\cO_{k,S})$
then the completion of $\Gamma$ w.r.t. the family $\{K_{I}\cap\Gamma\}_{I\vartriangleleft\cO_{k,S}}$
is called the congruence completion of $\Gamma$ and is denoted by
$\ov{\Gamma}$. One has a similar short exact sequence

\[
1\ra C(S,\Gamma)\ra\widehat{\Gamma}\ra\ov{\Gamma}\ra1.
\]

\begin{defn}
We say that $\Gamma\in\cA$ has the congruence subgroup property if
$|C(S,\Gamma)|<\infty$, and we denote by $\cA_{CSP}$ the set of
elements of $\cA$ admitting the congruence subgroup property.\end{defn}
\begin{rem}
We used Lemma \ref{lem:finiteindex free of compacts} to show that
we can assume that for any $\Gamma=(\Gamma,G,k,S)\in\cA$, $S$ does
not contain a valuation $v$ for which $G(k_{v})$ is compact. This
assumption affect what we mean by {}``the congruence completion''.
As first noted by Raghunathan, when $S$ contain a valuation $v$
for which $G(k_{v})$ is compact, $G(\cO_{k,S})$ has finite-index
subgroups corresponding to its embedding in $G(k_{v})$ that are not
consider traditionally as congruence subgroups. For example, given
an element of the form $G(\bZ)\in\cA$ such that $G(\bQ_{p_{0}})$
is compact, we've seen in Lemma \ref{lem:finiteindex free of compacts}
that $G(\bZ)<_{fi}G(\bZ[\frac{1}{p_{0}}])$. Considering $G(\bZ[\frac{1}{p_{0}}])$
with the ambient information $(G,\bQ,\{\infty,p_{0}\})$ the congruence
completion of $G(\bZ[\frac{1}{p_{0}}])$ is equal to $\prod_{q\neq p_{0}}G(\bZ_{q})$.
On the other hand, as $G(\bZ)<_{fi}G(\bZ[\frac{1}{p_{0}}])$ we can
also consider $G(\bZ[\frac{1}{p_{0}}])$ with the ambient information
$(G,\bQ,\{\infty\})$; the congruence completion of $G(\bZ[\frac{1}{p_{0}}])$
in this case is equal to $G(\bQ_{p})\times\prod_{q\neq p_{0}}G(\bZ_{q})$.
See \cite[Page 42]{SW09} for more details and for related notion
of the {}``adelic completion''.
\end{rem}

\subsection{A number theoretic Lemma}

We record a consequence of Chebotarev density Theorem which characterize
the normal closure of a number field in terms of its splitting primes:
\begin{lem}
\label{Lemma split primes} Let $Spl(k)$ be the set rational primes
which are of totally split in $k$. If the symmetric difference of
$Spl(k)$ with $Spl(k')$ is a finite set then $k$ and $k'$ have
the same normal closure. \end{lem}
\begin{proof}
See \cite[Theorem 43 and Chapter 8 exercise 1]{Ma77}. 
\end{proof}

\section{Finiteness up-to commensurability}

\label{sec: reductions up-to commensurabilty} In this section we
prove Theorem \ref{Main Theorem 1}. The main step in our proof is
to gather enough information about the following Lie algebras: Given
$(\Gamma,G,k,S)\in\cA$, for any prime number $p$ let $L_{p}^{\Gamma}$
be the Lie algebra of 
\[
\prod_{v\notin S,v|p}G(k_{v})
\]
 considered over $\bQ_{p}$. The following proposition shows that
these Lie algebras are encoded in the group structure of the profinite
completion:
\begin{prop}
\label{prop:lie iso SR} Let $\Lambda\in\cA_{\Gamma}$. Then $L_{p}^{\Gamma}\cong L_{p}^{\Lambda}$
for all primes $p$.
\end{prop}
We prove Proposition \ref{prop:lie iso SR} at the end of this section
using Margulis Super-Rigidity Theorem. We will first deduce Theorem
\ref{Main Theorem 1} from Proposition \ref{prop:lie iso SR} by extracting
arithmetic and algebraic information from the Lie algebras appearing
in Proposition \ref{prop:lie iso SR}.

\subsection{The Lie algebra $L_{p}^{\Gamma}$}

In this subsection we collect the arithmetic and geometric information
that can be deduced from knowing $L_{p}^{\Gamma}$ for all $p$.
\begin{lem}
\label{lem:Lie algebra} Let $(\Gamma,G,k,S)\in\cA$. Then, 
\begin{enumerate}
\item \label{lemitemRes:Lie algebra-1}The Lie algebra $L_{p}^{\Gamma}$
is semi-simple. Its simple components are $Lie_{\bQ_{p}}G(k_{v})$
for all $v\notin S,v|p$. 
\item \label{Lem lie algebra- simple comp} Let $n_{p}^{\Gamma}$ be the
number of simple components of $L_{p}^{\Gamma}$, and let $N_{\Gamma}:=max_{p}\{n_{p}^{\Gamma}\}$.
Then $n_{p}^{\Gamma}$ equals the number of valuations $v\notin S$
that lies over $p$. In particular, $N_{\Gamma}=[k:\bQ]$. Furthermore,
$n_{p}^{\Gamma}=N_{\Gamma}$ if and only if $p$ is a full prime for
$\Gamma$ that splits completely in $k$. 
\item \label{Lem lie algebra- restriction} For a full prime $p$, the Lie
algebra $L_{p}^{\Gamma}$ is naturally isomorphic to the the $\bQ_{p}$-Lie
algebra of $ $${\rm Res_{k/\bQ}G(\bQ_{p})}$. 
\item \label{lemitemRes:Lie algebra-4}Let $D_{p}^{\Gamma}:=dim_{\bQ_{p}}(L_{p}^{\Gamma})$
, $D_{\Gamma}=max_{p}(D_{p}^{\Gamma})$ and $dim(G):=dim_{\bC}(G(\bC))$.
Then 
\[
D_{p}^{\Gamma}=dim(G)\cdot\sum_{v\notin S,v|p}[k_{v}:\bQ_{p}]\quad\text{and }\quad D_{\Gamma}=dim(G)\cdot[k:\bQ].
\]
 Furthermore, $p$ is full if and only if $D_{p}^{\Gamma}=D_{\Gamma}$. 
\end{enumerate}
\end{lem}
\begin{proof}
We write $Lie_{k_{v}}(G(k_{v}))$ for the Lie algebra $Lie(G(k_{v}))$
considered over $k_{v}$ and $Lie_{\bQ_{p}}(G(k_{v}))$ for the same
Lie algebra but considered over $\bQ_{p}$. As $G$ is absolutely
simple, $Lie_{k_{v}}(G(k_{v}))$ is simple. As $[k_{v}:\bQ_{p}]$
is finite, it is shown in \cite[\S 6.10]{bour98} that $Lie_{\bQ_{p}}(G(k_{v}))$
is also simple, so the claims of (\ref{lemitemRes:Lie algebra-1})
follow. This also shows that each valuation $v\notin S$ that lies
over $p$ corresponds to a simple components of $L_{p}^{\Gamma}$.
Thus, using the fundamental identity $[k:\bQ]=\sum_{v\in V^{k},v|p}[k_{v}:\bQ_{p}]$,
the further claims of (\ref{Lem lie algebra- simple comp}) also follow.
The assertion in (\ref{Lem lie algebra- restriction}) follow directly
from the definition of restriction of scalars and proved in \cite[2.1.2]{PR94}.

For (\ref{lemitemRes:Lie algebra-4}), first note that as $G$ is
absolutely simple, we know that for any valuation $v$, $dim_{k_{v}}(Lie_{k_{v}}(G(k_{v})))=dim(G)$.
So, 
\[
dim_{\bQ_{p}}(Lie_{\bQ_{p}}(G(k_{v})))=[k_{v}:\bQ_{p}]\cdot dim_{k_{v}}(Lie_{k_{v}}(G(k_{v})))=[k_{v}:\bQ_{p}]\cdot dim(G).
\]
Using again the identity $[k:\bQ]=\sum_{v\in V^{k},v|p}[k_{v}:\bQ_{p}]$,
we see that 
\[
D_{p}^{\Gamma}=dim(G)\cdot\sum_{v\notin S,v|p}[k_{v}:\bQ_{p}]\leq dim(G)\cdot[k:\bQ],
\]
with equality if and only if $p$ is full. The latter bound does not
depend on $p$ so $D_{\Gamma}=dim(G)\cdot[k:\bQ]$.\end{proof}
\begin{cor}
\label{cor:dimension and normal closures} Fix $\Gamma\in\cA$ . Then,
for all $\Lambda\in\cA_{\Gamma}$ we have 
\begin{enumerate}
\item \label{enu:normal closure}$[k_{\Gamma}:\bQ]=[k_{\Lambda}:\bQ]$,
and the normal closure of $k_{\Gamma}$ is equal to the normal closure
of $k_{\Lambda}$, 
\item \label{enu:dimension}$dim(G_{\Gamma})=dim(G_{\Lambda})$,
\item \label{enu:S}$S_{\Gamma}^{full}=S_{\Lambda}^{full}$,
\item \label{enu:restriction}For all full primes $p$, we have $Res_{k_{\Lambda}/\bQ}(G_{\Lambda})\cong Res_{k_{\Gamma}/\bQ}(G_{\Gamma})$
as algebraic groups over $\bQ_{p}$.
\end{enumerate}
\end{cor}
\begin{proof}
We use the results and notation of Lemma \ref{lem:Lie algebra} freely
in this proof. From Proposition \ref{prop:lie iso SR} we know that
for all $p$, $n_{p}^{\Gamma}=n_{p}^{\Lambda}$ and $D_{p}^{\Gamma}=D_{p}^{\Lambda}$.
We immediately get 
\[
[k_{\Gamma}:\bQ]=N_{\Gamma}=N_{\Lambda}=[k_{\Lambda}:\bQ]
\]
 and likewise 
\[
[k_{\Gamma}:\bQ]\cdot dim(G_{\Gamma})=D_{\Gamma}=D_{\Lambda}=[k_{\Lambda}:\bQ]\cdot dim(G_{\Lambda})
\]
 so $dim(G_{\Gamma})=dim(G_{\Lambda})$. Furthermore, 
\[
S_{\Gamma}^{full}=\{p:D_{p}^{\Gamma}=D_{\Gamma}\}=\{p:D_{p}^{\Lambda}=D_{\Lambda}\}=S_{\Lambda}^{full}.
\]
We now show that $k_{\Gamma}$ and $k_{\Lambda}$ has the same normal
closure. Let $A$ be the set of split primes in $k_{\Gamma}$ and
$B$ be the set of split primes in $k_{\Lambda}$. Recall that a full
prime $p$ splits completely in $k_{\Gamma}$ if and only if $n_{p}^{\Gamma}=N_{\Gamma}$.
Clearly, $n_{p}^{\Gamma}=N_{\Gamma}$ if and only if $n_{p}^{\Lambda}=N_{\Lambda}$.
Thus, a \emph{full} prime $p$ splits completely in $k_{\Gamma}$
if and only if it splits completely in $k_{\Lambda}$. Therefore,
the symmetric difference of $A$ and $B$ is contained in the complement
of $S_{\Gamma}^{full}$, which is a finite set. By Lemma \ref{Lemma split primes},
$k_{\Gamma}$ and $k_{\Lambda}$ has the same normal closure. Finally,
for part (\ref{enu:restriction}), let $p$ be a full prime and note
that from Lemma \ref{lem:Lie algebra}.\ref{Lem lie algebra- restriction}
we have 
\[
Lie(Res_{k_{\Lambda}/\bQ}(G_{\Lambda})(\bQ_{p}))\cong L_{p}^{\Lambda}\cong L_{p}^{\Gamma}\cong Lie(Res_{k_{\Gamma}/\bQ}(G_{\Gamma})(\bQ_{p})).
\]
 As both groups are simply connected, this Lie algebra isomorphism
implies that the groups are isomorphic over $\bQ_{p}$, as claimed.
\end{proof}

\subsection{Proof of Theorem \ref{Main Theorem 1} assuming Proposition \ref{prop:lie iso SR}.}
\begin{proof}
Let us fix $\Gamma=(\Gamma,G,k,S)\in\cA$ and denote the elements
of $\cA_{\Gamma}$ by 
\[
\cA_{\Gamma}=\{(\Lambda_{i},G_{i},k_{i},S_{i})\}_{i\in I}.
\]
We advise the reader to recall the conventions and definitions from
Section \ref{sub:working-definition A}. For brevity, we say that
a subset of $\cA_{\Gamma}$ is CFIN if it is contained in finitely
many commensurability classes. Thus our aim is to show that $\cA_{\Gamma}$
itself is CFIN. By Corollary \ref{cor:dimension and normal closures}.\ref{enu:normal closure},
$\{k_{i}\}_{i\in I}$ is a finite set of fields as all of them have
the same normal closure. Thus we can divide $\cA_{\Gamma}$ into finitely
many subsets, according to their fields of definition. It is enough
to prove that each such subset is CFIN. 

Let $C=\{(\Lambda_{i},G_{i},k,S_{i})\}_{i\in I_{1}}$ be such a subset.
Part \ref{enu:S} of Corollary \ref{cor:dimension and normal closures},
implies that for any $i\in I_{1}$, $S_{\Gamma}^{full}=S_{i}^{full}$.
This imply that 
\[
S_{i}\subset\{v\in V^{k}:v|p\notin S_{\Gamma}^{full}\},\quad i\in I_{1}.
\]
Note that as $\Gamma$ is fixed throughout, $S_{\Gamma}^{full}$ is
a fixed set of rational primes, so $\{S_{i}\}_{i\in I_{1}}$ is a
finite set (of subsets of $V^{k}$). Thus, we can divide $C$ into
finitely many subsets, according to the sets of the corresponding
valuations. It is enough to prove that each such subset is CFIN . 

Let $D=\{(\Lambda_{i},G_{i},k,S)\}_{i\in I_{2}}$ be such a subset.
Corollary \ref{cor:dimension and normal closures}.\ref{enu:dimension}
shows that for any $i\in I_{2}$, $dim(G_{i})=dim(G)$. As there are
only finitely many semisimple groups of given dimension over $\bC$,
we deduce that the groups $G_{i}$ become isomorphic over the algebraic
closure to one of finitely many algebraic groups%
\footnote{In fact, one can easily see that \emph{all} of them become isomorphic
to $G(\bar{k})$ but we do not use this fact.%
}. We can divide $D$ into finitely many subsets, according to the
ambient groups. It is enough to prove that each such subset is CFIN. 

Let $E=\{(\Lambda_{i},G_{i},k,S)\}_{i\in I_{3}}$ be such a subset.
Corollary \ref{cor:dimension and normal closures}.\ref{enu:restriction}
implies that for any $i,j\in I_{3}$ we have $Res_{k/\bQ}(G_{i})(\bQ_{p})\cong Res_{k/\bQ}(G_{j})(\bQ_{p})$
for all $p\in S_{\Gamma}^{full}$. All but finitely many primes are
in $S_{\Gamma}^{full}$. Thus, all the groups in $\{Res_{k/\bQ}(G_{i})\}_{i\in I_{3}}$
agree locally in all but finitely many places. Using \cite{BS64},
this shows that there are finitely many $\bQ$-isomorphism classes
of $\bQ$-algebraic groups in $\{Res_{k/\bQ}(G_{i})\}_{i\in I_{3}}$.
Let $F:=\{(\Lambda_{i},G_{i},k,S)\}_{i\in I_{4}}\subset E$ such that
the elements of $\{Res_{k/\bQ}(G_{i})\}_{i\in I_{4}}$ belong to the
same $\bQ$-isomorphism class. We claim that all the elements of $\{\Lambda_{i}\}_{i\in I_{4}}$
belong to the same commensurability class. Indeed, by general properties
of the restriction of scalars functor, for any $i,j\in I_{4}$, the
isomorphism $Res_{k/\bQ}(G_{i})\cong Res_{k/\bQ}(G_{j})$ comes from
a unique $k$-isomorphism of $G_{i}$ with $G_{j}$ over a unique
automorphism of the field $k$. The reader can check (see \cite[Proposition 4.1]{PR94})
that the image of $\Lambda_{i}$ under such $k$-isomorphism and under
an automorphism of $k$ is a subgroup of $G_{j}(\bar{k})$ which is
commensurable to $\Lambda_{j}$. This shows that all the elements
of $F$ are contained in one commensurability class and thus $E$
is CFIN. By the reductions made above, this concludes the proof. $ $ 
\end{proof}

\subsection{Proof of Proposition \ref{prop:lie iso SR}}

As seen above, Proposition \ref{prop:lie iso SR} is the main ingredient
in the proof of Theorem \ref{Main Theorem 1}. The main idea of its
proof is that given $\Gamma\in\cA$, we will show that $L_{p}^{\Gamma}$
is the Lie algebra of a maximal $p$-adic analytic quotient of $\widehat{\Gamma}$.
For example, for $\Gamma=SL_{n}(\cO_{k})$ where $k$ is not a totally
imaginary number field, one has $\widehat{\Gamma}=\prod_{v\in V^{k}}SL_{n}(\cO_{v})$
and $\prod_{v|p}SL_{n}(\cO_{v})$ is a $p$-adic analytic quotient
of $\widehat{\Gamma}$ having $L_{p}^{\Gamma}$ as its Lie algebra.
The reader can check that it is also \emph{maximal} $p$-adic analytic
quotient of $\widehat{\Gamma}$ in the sense we will define momentarily.

The above example uses the congruence subgroup property, but the following
proof does not. The substitute is the fact that an element $\Gamma\in\cA$
satisfy the so-called Marguils' super-rigidity theorem \cite[Theorem 6, page 5]{Mar91}: 
\begin{thm}
{[}Margulis' super rigidity{]}\label{thm: superrigidity} Let $\Gamma\in\cA$
and $H$ be an algebraic group which is defined over a field $l$
with $char(l)=0$. Then for any homomorphism $\delta:\Gamma\ra H(l)$
there exist a finite index subgroup $\Gamma'<\Gamma$ such that $\delta|_{\Gamma'}$
factor as 
\[
\xymatrix{\Gamma'\ar[rrr]^{\delta|_{\Gamma'}}\ar@{^{(}->}[dr] &  &  & H(l)\\
 & G(k)\ar@{=}[r] & (Res_{k/\bQ}G)(\bQ)\ar[ur]^{\phi}
}
\]
 where $\phi$ is a (uniquely determined) $l$-morphism of algebraic
groups between $Res_{k/\bQ}G$ to the Zariski closure of $\delta(\Gamma')$. \end{thm}
\begin{proof}
Elements of $\cA$ satisfy the conditions appearing in \cite[Theorem 6, page 5]{Mar91}
which states the above claim.
\end{proof}
Fix a prime number $p$ and let $\cN(\Gamma,p)$ be the set of normal
subgroups $N$ such that $\widehat{\Gamma}/N$ is a $p$-adic analytic
group. Note that any such $p$-adic analytic group is a finitely generated
group which can be realized as a compact open subgroup of some algebraic
group defined of $\bQ_{p}$. For any $N\in\cN(\Gamma,p$) we fix such
a realization and identify $\widehat{\Gamma}/N$ with this realization.
We call an element $N\in\cN(\Gamma,p)$ \emph{maximal} if for any
$M\in\cN(\Gamma,p)$, 
\[
dim(\widehat{\Gamma}/N)\geq dim(\widehat{\Gamma}/M)
\]
where $dim$ is the dimension as a $p$-adic analytic group (i.e.,
the dimension of the underlying analytic manifold). We denote the
set of maximal elements of $\cN(\Gamma,p)$ by $\cN_{0}(\Gamma,p)$.
Let $(\Gamma,G,k,S)\in\cA$ and $\pi:\widehat{\Gamma}\ra\bar{\Gamma}$
be the natural map between the profinite completion and the congruence
completion of $\Gamma$. The group $\bar{\Gamma}$ is an open compact
subgroup of the Adele group $G(\bA_{k,S})$. Let $\pi_{p}$ be the
projection from $G(\bA_{k,S})$ to $\prod_{v\notin S,v|p}G(k_{v})$.
\begin{prop}
\label{prop:maximal} Let $(\Gamma,G,k,S)\in\cA$, $N_{p}:=ker(\pi_{p}\circ\pi)$.
Then $N_{p}\in\cN_{0}(\Gamma,p)$. Moreover, for any $N\in\cN_{0}(\Gamma,p)$
there exist an isomorphism of Lie algebras between $Lie(\widehat{\Gamma}/N)$
and $Lie(\widehat{\Gamma}/N_{p})$. \end{prop}
\begin{proof}
It is clear that $N_{p}\in\cN(\Gamma,p)$ and also the quotient $\widehat{\Gamma}/N_{p}$
is naturally contained in 
\[
\prod_{v|p,v\notin S}G(k_{v})\cong\prod_{v|p,v\notin S}(Res_{k_{v}/\bQ_{p}}G)(\bQ_{p})
\]
 and commensurable to 
\[
D:=\prod_{v|p,v\notin S}G(\cO_{v})\cong\prod_{v|p,v\notin S}(Res_{k_{v}/\bQ_{p}}G)(\bZ_{p}).
\]
These properties of the restriction of scalars functor are proved
in \cite[2.1.2]{PR94}. 

Now, in order to show that $N_{p}\in\cN_{0}(\Gamma,p)$, we need to
verify the maximality condition. As $\widehat{\Gamma}/N_{p}$ is commensurable
to $D$, they have the same dimension as $p$- adic analytic groups.
The group $D$ is open in $\prod_{v|p,v\notin S}G(k_{v})$, and its
dimension as a $p$-adic analytic group is equal to the dimension
over $\bQ_{p}$ of the group $\prod_{v|p,v\notin S}G(k_{v})$. The
latter is equal to $\sum_{v|p,v\notin S}[k_{v}:\bQ_{p}]\cdot dim(G)$.
Note that when $p\in S_{\Gamma}^{full}$, this dimension is equal
to $[k:\bQ]\cdot dim(G)$. 

We now use Margulis' super-rigidity (Theorem \ref{thm: superrigidity})
to show that the dimension of $\widehat{\Gamma}/N_{p}$ is the maximal
dimension of a $p$-adic analytic quotient of $\widehat{\Gamma}$
that may occur. Let $M\in\cN(\Gamma,p)$, i.e., $M\vartriangleleft\widehat{\Gamma}$
such that $\widehat{\Gamma}/M$ is a $p$-adic analytic group which
can be realized as an open compact subgroup of $H(\bQ_{p})$ where
$H$ is an algebraic group which is defined over $\bQ_{p}$; Moreover,
$\widehat{\Gamma}/M$ is and Zariski dense in $H(\bQ_{p})$. The openness
condition implies that the dimension of $H$ as an algebraic group
is equal to the dimension of $\widehat{\Gamma}/M$ as a $p$-adic
analytic subgroup, so our goal is to show that 
\begin{equation}
dim(H)\leq\sum_{v\notin S}[k_{v}:\bQ_{p}]\cdot dim(G).\label{eq:dimension goal}
\end{equation}
 By Theorem \ref{thm: superrigidity} the map 
\[
\Gamma\ra\widehat{\Gamma}\ra\widehat{\Gamma}/M\subset H(\bQ_{p})
\]
 factor, on a finite-index subgroup $\Gamma'$ of $\Gamma$, as 
\[
\xymatrix{\Gamma'\ar@{^{(}->}[d]\ar@{^{(}->}[r] & \widehat{\Gamma}\ar[rr] &  & \widehat{\Gamma}/M\subset H(\bQ_{p})\\
G(k)\ar@{=}[r] & Res_{k/\bQ}G(\bQ)\ar@{^{(}->}[r] & Res_{k/\bQ}G(\bQ_{p})\ar@{>>}[ur]^{\phi}
}
\]
 where $\phi:Res_{k/\bQ}(G)\ra H$ is a surjective $\bQ_{p}$-morphism
of algebraic groups. Note that for $p\in S_{\Gamma}^{full}$, the
desired inequality in (\ref{eq:dimension goal}) follows as the surjectivity
of $\phi$ yields that 
\[
dim(H)\leq dim(Res_{k/\bQ}(G))=[k:\bQ]\cdot dim(G)=dim(\widehat{\Gamma}/N_{p}).
\]
 For the general case, note that 

\begin{equation}
(Res_{k/\bQ}G)(\bQ_{p})\cong\prod_{v|p}(Res_{k_{v}/\bQ_{p}}G)(\bQ_{p})\cong\prod_{v|p}G(k_{v})\label{eq:Qp points}
\end{equation}
and we conclude by showing that $\phi$ induces a surjective map from
\[
\prod_{v|p,v\notin S}(Res_{k_{v}/\bQ_{p}}G)\rightarrow H,
\]
which gives the desired inequality (\ref{eq:dimension goal}) as above. 

Let $C$ be the closure of $\Gamma'\cap G(\cO_{k,S})$ in (\ref{eq:Qp points}).
We claim that $C=C_{1}\times C_{2}$ where 
\[
C_{1}<_{fi}\prod_{v|p,v\notin S}G(\cO_{v})=\prod_{v|p,v\notin S}(Res_{k_{v}/\bQ_{p}}G)(\bZ_{p})
\]
\[
C_{2}=\prod_{v|p,v\in S}G(k_{v})=\prod_{v|p,v\in S}(Res_{k_{v}/\bQ_{p}}G)(\bQ_{p})
\]

Indeed, by our assumption on the $S$-rank of $G$, $G$ satisfy the
strong approximation property w.r.t. $S$ (see \cite[Theorem 7.12 and Proposition 7.2(2)]{PR94})
which assert the above. Note that $C_{1}$ is compact and $C_{2}$
is an affine semisimple algebraic groups without any compact factor,
by the assumption made on $S$ after Lemma \ref{lem:finiteindex free of compacts}.
The map $\phi$ induces a map from $C$ to the compact group $\widehat{\Gamma}/M$,
and the following Lemma shows that $\{e\}\times C_{2}$ is mapped
to the identity of the compact group $\widehat{\Gamma}/M$.
\begin{lem}
\label{lem: map to compact-1} Let L be a linear algebraic group such
that there exist a finite collection of one-dimensional unipotent
subgroups $\{U_{i}\}_{i=1}^{l}$ with 
\[
L(\bQ_{p})=U_{1}(\bQ_{p})\cdots U_{l}(\bQ_{p}).
\]
Then, any polynomial map $\varphi$ from $G(\bQ_{p})$ to a compact
variety is a constant map.\end{lem}
\begin{proof}
{[}Proof of Lemma \ref{lem: map to compact-1}{]} Using the ultrametric
on $\bQ_{p}$ one can easily show that any polynomial $f(x)$ with
coefficients in $\bQ_{p}$ is either constant or unbounded. This imply
that $\varphi:=\varphi|_{U_{i}(\bQ_{p})}$ is a constant map for $i=1,\dots,l$.
As all the maps $\varphi_{i}$ agree on the identity, they agree everywhere.
Finally, since $L(\bQ_{p})=U_{1}(\bQ_{p})\cdots U_{l}(\bQ_{p})$,
$\varphi$ is the constant map.
\end{proof}
By \cite[Thm 27.5(e) \& Propositon 7.5(b)]{Hump75} $\{e\}\times C_{2}$
satisfy the conditions of Lemma \ref{lem: map to compact-1} with
$L=C_{2}$ and the map $\phi$ (which is a $\bQ_{p}$-polynomial map).
Thus $\{e\}\times C_{2}$ is in the kernel of $\phi$. Dividing by
$\{e\}\times C_{2}$, $\phi$ induces a surjective morphism 
\[
\tilde{\phi}:\prod_{v|p,v\notin S}Res_{k_{v}/\bQ_{p}}G\ra H
\]
 so $dim(H)$ cannot exceed 
\[
dim_{\bQ_{p}}(\prod_{v|p,v\notin S}Res_{k_{v}/\bQ_{p}}G)=\sum_{v\notin S}[k_{v}:\bQ_{p}]\cdot dim(G).
\]
 as desired. 

The above shows that taking $M=N$ for some $N\in\cN_{0}(\Gamma,p)$
we get a surjective morphism between $\tilde{\phi}:\prod_{v|p,v\notin S}Res_{k_{v}/\bQ_{p}}G\ra H$
and from the maximality condition, $H$ has the same dimension as
$Res_{k_{v}/\bQ_{p}}G$. Thus $\tilde{\phi}$ induces an isomorphism
between $Lie(Res_{k_{v}/\bQ_{p}}G)=Lie(\widehat{\Gamma}/N_{p})$ and
$Lie(H)=Lie(\widehat{\Gamma}/N)$, as claimed.
\end{proof}
We now finally prove Proposition \ref{prop:lie iso SR}:
\begin{proof}
Let $\Phi:\widehat{\Gamma}\ra\widehat{\Lambda}$ be a given isomorphism,
$p$ a prime number and $N\in\cN_{0}(\Gamma,p)$. Note that $L_{p}^{\Gamma}=Lie(\widehat{\Gamma}/N_{p})$
and therefore by Proposition \ref{prop:maximal} $L_{p}^{\Gamma}=Lie(\widehat{\Gamma}/N)$
for any $N\in\cN_{0}(\Gamma,p)$. The reader can easily verify that
assigning to each subgroup of $\widehat{\Gamma}$ its image by $\Phi$,
we get a bijection between $\cN(\Gamma,p)$ and $\cN(\Lambda,p)$
which restricts to a bijection between $\cN_{0}(\Gamma,p)$ and $\cN_{0}(\Lambda,p)$.
Moreover, it is clear that for any $N\vartriangleleft\widehat{\Gamma},$
we have $\widehat{\Gamma}/N\cong\widehat{\Lambda}/\Phi(N)$. Therefore
\[
L_{p}^{\Gamma}=Lie_{\bQ_{p}}(\widehat{\Gamma}/N)\cong Lie_{\bQ_{p}}(\widehat{\Lambda}/\Phi(N))=L_{p}^{\Lambda}.
\]

\end{proof}

\section{Examples\label{sec:Examples}}

In the course of the proof of Theorem \ref{Main Theorem 1} we deduced
from $ $$\widehat{\Gamma}\cong\widehat{\Lambda}$ the following facts:
the groups $G_{\Gamma}$ and $G_{\Lambda}$ are forms of each other,
the fields $F_{\Gamma}$ and $F_{\Lambda}$ have the same normal closure
and the sets $S_{\Gamma}$ and $S_{\Lambda}$ have the same set of
full primes.

The purpose of this section is to demonstrate that all the steps of
the proof of Theorem \ref{Main Theorem 1} are reflected in non-trivial
examples. This is done by giving explicit examples of non-isomorphic
arithmetic groups which are profinitely isomorphic. The examples in
\ref{sub:exa_Real_forms} and \ref{sub:diff S's} exhibit an explicit
family of groups for which, in terms of \cite{GZ11}, have unbounded
genus, hence give an answer to problem III of \cite{GZ11}.

\subsection{Real forms, Profinite properties and Kahzdan's property (T)\label{sub:exa_Real_forms}}

A property $\cP$ of finitely generated residually finite groups is
called a \emph{profinite property} if the following is satisfied:
if $\Gamma_{1}$ and $\Gamma_{2}$ are such groups with $\widehat{\Gamma_{1}}\cong\widehat{\Gamma_{2}}$
then $\Gamma_{1}$ has $\cP$ if and only if $\Gamma_{2}$ has $\cP$.
See \cite{Aka10} and the references within for various interesting
profinite properties and various non-profinite properties. Kassabov
showed that that property $(\tau)$ is not a profinite property, and
he asked whether Kazhdan property (T) is a profinite property. In
\cite{Aka10}, it is shown that Kazhdan property (T) is not a profinite
property. Explicitly, the following is proved: 
\begin{thm}
\label{thm:main theorem} Let $D$ be a positive square-free integer,
$k:=\bQ(\sqrt{D})$ and $\cO_{k}$ its ring of integers. Fix an integer
$n>7$ and let $\Gamma=Spin(1,n)(\cO_{k})$ and $\Lambda=Spin(5,n-4)(\cO_{k})$.
Then, there exist finite-index subgroups $\Gamma_{0}<\Gamma$ and
$\Lambda_{0}<\Lambda$ such that the profinite completion of $\Gamma_{0}$
is isomorphic to the profinite completion of $\Lambda_{0}$ while
$\Lambda_{0}$ admits property $(T)$ and $\Gamma_{0}$ does not.
Therefore, Kazhdan property (T) is not profinite. 
\end{thm}
In particular, there exist non-isomorphic arithmetic groups with isomorphic
profinite completions. Note that $Spin(1,n)(\cO_{k})$ (resp. $Spin(5,n-4)(\cO_{k})$)
are central extensions of irreducible lattices in $H_{1}:=SO(1,n)\times SO(1,n)$
(resp.$H_{2}:=SO(5,n-4)\times SO(5,n-4)$). We can therefore deduce
that $H_{1}$ and $H_{2}$ have lattices $\Gamma_{1}$ and $\Lambda_{1}$
resp. with isomorphic profinite completions. As $\bR-rank(H_{1})=2$
while $\bR-rank(H_{2})=10$ we see that the rank of the ambient Lie
group is also not a profinite property (in contrast to well know rigidity
results of Raghunathan and Margulis which showed the ambient algebraic
group of isomorphic lattices have the rank).

Note that $G_{\Gamma}=Spin(1,n)$ and $G_{\Lambda}=Spin(5,n-4)$ are
both forms of the complex lie group $Spin(n+1)$. All the other forms
are the groups $Spin(1+4k,n-4k)$ for suitable $k$. Choosing suitable
$n$ and $k$'s the same proof show that one can construct any given
number of non-isomorphic arithmetic groups with isomorphic profinite
completion. Thus this family of groups has unbounded genus so this
give a proof of Theorem \ref{thm:Unbounded genus}. In fact, one may
show that for any split $G$ of high rank, any non-trivial element
of $H^{1}(\bR,G)$ can give rise to a pair of non-isomorphic arithmetic
group with isomorphic profinite completion.

\subsection{Arithmetic equivalence and Adelic equivalence of fields\label{sub:Arithmetic-equivalence}}

Let $k$ be a number field, $\zeta_{k}$ its Dedekind zeta function
and $\bA_{k}$ its Adele ring. Recall also that $V^{k}$ is the set
of all valuations on $k$. We define two closely related equivalence
relations on number fields : $k$ and $l$ are said to be \emph{arithmetic
equivalent} if $\zeta_{k}=\zeta_{l}$, i.e., if their Dedekind zeta
functions are equal. Similarly, $k$ and $l$ are said to be \emph{Adelic
equivalent} if $\bA_{k}\cong\bA_{l}$, i.e., if their Adele rings
are isomorphic. The following proposition shows a strong relation
between these two notions: 
\begin{prop}
\label{prop:Adelic and Arith} 
\begin{enumerate}
\item Two number fields $k$ and $l$ are arithmetic equivalent if and only
if there exist $S\subset V^{k},T\subset V^{l}$ and a bijection $\phi:V^{k}\setminus S\lra V^{l}\setminus T$
such that for all $v\in V^{k}\setminus S$ there exist isomorphism
$k_{v}\cong l_{\phi(v)}$. 
\item Two number fields $k$ and $l$ are Adelic equivalent if and only
if there exist a bijection $\phi:V^{k}\ra V^{l}$ such that for all
$v\in V^{k}$ there exist isomorphism $k_{v}\cong l_{\phi(v)}$. 
\end{enumerate}
\end{prop}
\begin{proof}
See \cite[page 235 iv) and Theorem 2.3]{Klin98}.
\end{proof}
One clearly see that being Adelic equivalent is a stronger condition. 
\begin{thm}
There exist two non-isomorphic (and not totally imginary) number fields
$k$ and $l$ which are Adelic equivalent. Furthermore, they have
the same degree, the same set of splitting primes and the same normal
closure . \end{thm}
\begin{proof}
See \cite[page 239-240]{Klin98}. 
\end{proof}
We now show that non-isomorphic Adelic equivalent fields give a simple
example of non-isomorphic arithmetic groups with isomorphic profinite
completions. Similarly, non-isomorphic arithmetic equivalent fields
give a simple example of non-isomorphic $S$-arithmetic groups with
isomorphic profinite completions. 
\begin{prop}
Fix the standard representation of $SL_{n}$ ($n\geq3$) and let $k$
and $l$ be Adelic equivalent non-isomorphic fields which are not
totally imaginary. Set $\Gamma=SL_{n}(\cO_{k})$ and $\Lambda=SL_{n}(\cO_{l})$.
Then, $\widehat{\Gamma}=\widehat{\Lambda}$ and $\Gamma$ and $\Lambda$
are non-isomorphic. \end{prop}
\begin{proof}
As $k$ and $l$ are not totally imaginary, the congruence kernel
of $SL_{n}$ w.r.t. $k$ and $l$ we have that 
\[
\widehat{\Gamma}=\widehat{SL_{n}(\cO_{k})}=SL_{n}(\widehat{\cO_{k}})=\prod_{v\in V^{k}}SL_{n}(\cO_{v})
\]
 and similarly 
\[
\widehat{\Lambda}=\widehat{SL_{n}(\cO_{l})}=SL_{n}(\widehat{\cO_{l}})=\prod_{w\in V^{l}}SL_{n}(\cO_{w}).
\]

From the second part Proposition \ref{prop:Adelic and Arith} we know
that since $k$ and $l$ are Adelic equivalent, there exist a bijection
$\phi:V^{k}\ra V^{l}$ such that for all $v\in V_{k}$ there exist
isomorphism between $k_{v}$ and $l_{\phi(v)}$. This isomorphism
induces an isomorphism from $(\cO_{k})_{v}$ to $(\cO_{l})_{\phi(v)}$
and as $SL_{n}$ is defined over $\bQ$ it also induces isomorphism
between $SL_{n}((\cO_{k})_{v})$ to $SL_{n}((\cO_{l})_{\phi(v)})$.
This shows that $\widehat{\Gamma}\cong\widehat{\Lambda}$.

If there were an isomorphism $\Psi:\Gamma\ra\Lambda\subset SL_{n}(l)$,
another version of Margulis' super-rigidity \cite[Theorem 5, page 5]{Mar91}
implies that there exist an embedding of $k$ into $l$, which contradicts
our assumptions. 
\end{proof}
We note that using the first part of Proposition \ref{prop:Adelic and Arith},
the same proof will yield an example of non-isomorphic $S$-arithmetic
groups with isomorphism profinite completion from non-isomorphic arithmetic
equivalent fields. Moreover, using the same group theoretic techniques
that yield arithmetic equivalent pairs of fields, one can find unboundedly
many fields which are arithmetic equivalent. This in turn with give
rise to a family of arithmetic groups with unbounded genus as claimed
in Theorem \ref{thm:Unbounded genus}.

\subsection{Examples involving different $S$'s.\label{sub:diff S's}}

Let $k/\bQ$ be a Galois extension of degree $n$, $p_{0}$ be a prime
that splits completely in $k$ and $X_{p_{0}}:=\{v_{1},\dots,v_{n}\}$
denote the set of all valuation over $p_{0}$. Recall that any element
$\sigma\in Gal(k/\bQ)$ induces a permutation of $X_{p_{0}}$ . For
each $I\subset\{v_{1},\dots,v_{n}\}$ let $\Gamma_{I}:=SL_{d}(\cO_{k,I\cup V_{\infty}^{k}})$
where $d\geq3$ .
\begin{prop}
We have $\widehat{\Gamma_{I}}\cong\widehat{\Gamma_{J}}$ if and only
if $|I|=|J|$.\end{prop}
\begin{proof}
As $d\geq3$, $\Gamma_{I}$ posses the congruence subgroup property
for any $I$. This implies that 
\[
\widehat{\Gamma_{I}}=\prod_{v\notin I}SL_{d}(\cO_{v})=\prod_{v\in X_{p_{0}}\setminus I}SL_{d}(\bZ_{p_{0}})\times\prod_{p\neq p_{0}}\prod_{v|p}SL_{d}(\cO_{v})
\]
since the assumption that $p_{0}$ splits completely implies that
$\cO_{v}=\bZ_{p_{0}}$ for all $v\in X_{p_{0}}$. This clearly shows
that $\widehat{\Gamma_{I}}\cong\widehat{\Gamma_{J}}$ whenever $|I|=|J|$.
The {}``only if'' part follows from Proposition \ref{prop:lie iso SR};
indeed, $\widehat{\Gamma_{I}}\cong\widehat{\Gamma_{J}}$ implies that
$L_{p_{0}}^{\Gamma_{I}}\cong L_{p_{0}}^{\Gamma_{J}}$ and these Lie
algebras are isomorphic if and only if $|X_{p_{0}}\setminus I|=|X_{p_{0}}\setminus J|\Longleftrightarrow|I|=|J|$
. 
\end{proof}
Margulis' super rigidity tells us when the arithmetic groups themselves,
rather than merely their profinite completion, can be isomorphic:
\begin{thm}
We have $\Gamma_{I}\cong\Gamma_{J}$ if and only if there exist $\sigma\in Gal(k/\bQ)$
with $\sigma(I)=J$.\end{thm}
\begin{proof}
Acting on the matrix entries (in the standard representation) of $\Gamma_{I}$
with $\sigma$ gives an isomorphism between $\Gamma_{I}$ and $\Gamma_{\sigma(I)}$,
so the if part is obvious. Now assume $\Gamma_{I}\cong\Gamma_{J}$
and map $\Gamma_{J}$ diagonally to 
\[
Res_{k/\bQ_{p}}(SL_{n})(\bQ_{p_{0}})=\prod_{v\in X_{p_{0}}}SL_{n}(\bQ_{p_{0}}).
\]
Up to passing to finite-index subgroup of $\Gamma_{I}$, Margulis'
super rigidity gives us the following commutative diagram: 
\[
\xymatrix{\Gamma_{I}\ar@{^{(}->}[d]\ar@{^{(}->}[r] & \Gamma_{J}\ar[rr] &  & \overline{\Gamma_{J}}\subset Res_{k/\bQ}(SL_{d})(\bQ_{p_{0}})\\
SL_{d}(k)\ar@{=}[r] & Res_{k/\bQ}(SL_{d})(\bQ)\ar@{^{(}->}[r] & Res_{k/\bQ}(SL_{d})(\bQ_{p_{0}})\ar@{>>}[ur]^{\phi}
}
\]
where $\phi:Res_{k/\bQ}(SL_{d})\rightarrow Res_{k/\bQ}(SL_{d})$ is
a map of algebraic groups. By the universal property of the restriction
of scalars functor, $\phi$ comes from a unique automorphism of $SL_{d}$
composed with an automorphism $\sigma$ of $k/\bQ$. Moreover, it
must map the closure of $\Gamma_{I}$ to the closure of $\Gamma_{J}$.
Therefore, $\sigma$ must map $I$ to $J$, as claimed. 
\end{proof}
The following corollary gives another family of groups which prove
Theorem \ref{thm:Unbounded genus}.
\begin{cor}
Let $I$ be a subset of $X_{p}$ of size $l$ with $0<l\leq n$. Then,
the set $\cA_{\Gamma_{I}}$ has at least $\frac{{n \choose l}}{n}$
isomorphism classes and is contained in at least $\frac{{n \choose l}}{n}$
commensurability classes. 
\end{cor}

\section{Finiteness within a commensurability class}

We now turn to establish the desired finiteness result within a commensurability
class that admits the congruence subgroup property. Recall the definition
of a commensurability class given in subsection \ref{sub:working-definition A}.
Given a commensurability class $\cC$, there exist $G,k$ and $S$
such that each element of $\cC$ has an isomorphic copy in 

\[
\cB_{G,k,S}:=\{\Lambda\in\cA:(G_{\Lambda},F_{\Lambda},S_{\Lambda})=(G,k,S)\}.
\]
Note that all the elements of $\cB_{G,k,S}$ posses the congruence
subgroup property if and only if the congruence kernel of $G(\cO_{k,S})$
is finite. In this section we prove:
\begin{thm}
\label{thm: B family} Let $G(\cO_{k,S})$ be an element of $\cA_{CSP}$,
and $\cB:=\cB_{G,k,S}$. For all $\Gamma\in\cB$, $\cB_{\Gamma}:=\{\Lambda\in\cB:\widehat{\Lambda}\cong\widehat{\Gamma}\}$
is finite union of isomorphism classes. 
\end{thm}
We will in fact show that the covolume in $G_{S}:=\prod_{v\in S}G(k_{v})$
of elements in $\cB_{\Gamma}$ is equal to the covolume as $\Gamma$.
Using Theorem \ref{thm: B family} in conjunction with Theorem \ref{Main Theorem 1}
we prove Theorem \ref{thm:CSP} in the next section.

\subsection{Overview of the proof}

Given an isomorphism $\Phi:\widehat{\Gamma}\ra\widehat{\Lambda}$
, we will find in subsection \ref{sub:Factor-iso for subgps} $\Gamma^{0}<_{fi}\Gamma,$
$\Lambda^{0}<_{fi}\Lambda$ such that $\Phi|_{\widehat{\Gamma^{0}}}$
has special properties. We examine such maps in subsection \ref{sub:Product-type-elements}
and show in particular that these properties will allow us to deduce
that $\Gamma^{0}$ and $\Lambda^{0}$ has the same covolume in $G_{S}$.
From this we deduce the desired result in subsection \ref{sub:proof within}.

\subsection{Product-type elements and Factor-type isomorphisms\label{sub:Product-type-elements}}

The reader is advised to recall the notions from subsections \ref{sub:Adeles}
and \ref{sub:CSP} and that we called $\Gamma\in\cB$ \emph{rational}
if $\Gamma\subset G(k)$. We say that a rational element $\Gamma\in\cB$
is of \emph{product-type} if $\Gamma$ admits the congruence subgroup
property and if $\widehat{\Gamma}=\ov{\Gamma}=\prod_{v\notin S_{\Gamma}}\Gamma_{v}\subset G(\bA_{k,S})$
with $\Gamma_{v}\subset G(k_{v})$. For product-type element $\Gamma$,
we denote the $v$-factor of $\widehat{\Gamma}$ by $\Gamma_{v}$.
We let $\cB_{0}$ be the subset of elements of $\cB$ which are rational
and have the trivial group as their congruence kernel. As explained
in subsection \ref{sub:CSP} the profinite completion of an element
of $\cB_{0}$ is an open compact subgroups of $G(\bA_{k,S})$. 
\begin{defn}
\label{def:factor type} Let $\Gamma,\Lambda\in\cB$ be rational product-type
elements, i.e. $\Gamma,\Lambda\subset G(k)$ and $\widehat{\Gamma}=\prod_{v\notin S}\Gamma_{v}$,
$\widehat{\Lambda}=\prod_{v\notin S}\Lambda_{v}$. An isomorphism
$\Phi:\widehat{\Gamma}\ra\widehat{\Lambda}$ is called of \emph{factor-type}
if for every $v\notin S$, $\Phi|_{i_{v}(\Gamma_{v})}$ is an isomorphism
between $i_{v}(\Gamma_{v})$ and $i_{w}(\Lambda_{w})$ for some $w\notin S$. 
\end{defn}
We will shortly show that the existence of a factor-type isomorphism
implies the equality of the covolumes of $\Gamma$ and $\Lambda$
in $G_{S}$, and to this end we briefly review the volume formula
of Prasad. We need to use a very simple case of it which is supplied
by \cite[Theorem 3.7]{Pr89}. In the terms given there, a product-type
element $\Delta\in\cB_{0}$ is called the principal $S$-arithmetic
subgroup determined by the parahoric subgroups $\Delta_{v}$ where
$\widehat{\Delta}=\prod_{v\notin S}\Delta_{v}$ (See \cite[Theorem 2.1\&3.4]{Pr89}).
Theorem 3.7 of \cite{Pr89} shows that the volume of $G_{S}/{\Delta}$
is 
\[
vol(G_{S}/{\Delta})=C(G,k)\cdot\cE(\Delta)
\]
 where $C(G,k)$ is a constant that depends only on $G$ and $k$
and does not depend on $\Delta$, and $\cE(\Delta)$ is the infinite
product $\prod_{v\notin S}\epsilon_{v}(\Delta_{v})^{-1}$ where $\epsilon_{v}(\Delta_{v})$
is the volume of $\Delta_{v}$ with respect to the Haar measures $\omega_{v}^{*}$
which are normalized to give measure one to any Iwahori subgroup (See
\cite[3.4]{Pr89}). 
\begin{lem}
\label{lem:Covolume} Let $\Gamma,\Lambda\in\cB_{0}$ be product-type
elements (as in definition \ref{def:factor type}) such that there
exist a factor-type isomorphism between their profinite completions.
Then, $\Gamma$ and $\Lambda$ have the same covolume in $G_{S}:=\prod_{v\in S}G(k_{v})$.\end{lem}
\begin{proof}
We use the notation of Definition \ref{def:factor type}. As $\Gamma,\Lambda$
are both product-type elements of $\cB_{0}$, it is enough to show
that $\cE(\Gamma)=\prod_{v\notin S}\epsilon_{v}(\Gamma_{v})^{-1}$
is equal to $\cE(\Lambda)=\prod_{v\notin S}\epsilon_{v}(\Lambda_{v})^{-1}$.
The existence of a factor-type isomorphism implies that there exists
$\sigma\in Sym\{v:v\notin S\}$ such that $\Gamma_{v}$ is isomorphic
to $\Lambda_{\sigma(v)}$, for all $v\notin S$. As $G$ is absolutely
simple and simply-connected, and $k_{v},k_{\sigma(v)}$ are local
fields, a theorem of Pink (\cite[Theorem 0.3]{Pink98}) claims that
any such isomorphism extends uniquely to an algebraic automorphism
$\phi$ of $G$ over a unique isomorphism of local fields $\psi:k_{v}\ra k_{\sigma(v)}$.
We claim that such isomorphism is measure-preserving. Indeed, the
map $\psi$ is measure-preserving since the measures $\omega_{v}^{*}$
are normalized to give measure one to any Iwahori subgroup and $\psi$
maps an Iwahori subgroup to an Iwahori subgroup. Moreover, as $G$
is simple, it is unimodular, so the automorphism $\phi$ is also measure-preserving.
Therefore, $\epsilon_{v}(\Gamma_{v})=\epsilon_{\sigma(v)}(\Lambda_{\sigma(v)})$.
As $\sigma\in Sym\{v:v\notin S\}$ it follows that $\cE(\Gamma)=\cE(\Lambda)$,
so the covolumes of $\Gamma$ and $\Lambda$ are equal. 
\end{proof}
In the next subsection we will see that any isomorphism induces a
factor-type isomorphism on certain finite-index subgroups.

\subsection{\label{sub:Factor-iso for subgps}Factor-type isomorphism between
finite-index subgroups}

Fix a representation of $G$ to $GL_{n}$. By Jordan's Theorem (\cite[page 98]{Dix71})
there exist a constant $j=j(n)$ such that $\langle A^{j}\rangle$
(the group generated by the $j-th$ powers of A) is Abelian for any
finite subgroup $A$ of $GL_{n}(k_{v})$ (for any valuation $v\in V^{k}$).
The constant $j$ is called Jordan's constant. Recall that by Lemma
\ref{lem:subgroups of profinite completion}, we can define a finite-index
subgroup of an element $\Gamma$ by specifying a finite-index subgroup
of $\widehat{\Gamma}$. In this subsection we prove the following:
\begin{prop}
\label{prop:reduction to factor} Let $\Gamma,\Lambda$ be in $\cB_{0}$
and assume that $\Gamma$ is of product-type with $\widehat{\Gamma}=\prod_{v\notin S}\Gamma_{v}$.
Let $ $ $D=\prod_{v\notin S}\Gamma_{v}^{0}$ where $\Gamma_{v}^{0}=[\langle\Gamma_{v}^{j}\rangle,\langle\Gamma_{v}^{j}\rangle]$.
Then $D<_{fi}\widehat{\Gamma}$ and there exist $\Gamma^{0}<_{fi}\Gamma$
with $\widehat{\Gamma^{0}}=H$. Moreover, let $\Phi:\widehat{\Gamma}\ra\widehat{\Lambda}$
be an isomorphism and $\Lambda^{0}:=\Phi_{*}(\Gamma^{0})$ see (definition
\ref{def: FI corres}). Then, $\Lambda^{0}$ is of product-type and
$\Phi|_{{\widehat{\Gamma^{0}}}}$ is a factor-type isomorphism between
$\widehat{\Gamma^{0}}$ and $\Phi({\widehat{\Gamma^{0}}})=\widehat{\Lambda^{0}}$. \end{prop}
\begin{proof}
(of Proposition \ref{prop:reduction to factor}) We first show that
$D<_{fi}\widehat{\Gamma}$. It is obvious that the closure of $[\langle\Gamma^{j}\rangle,\langle\Gamma^{j}\rangle]$
in $\widehat{\Gamma}$ is contained in $D$. By Margulis' normal subgroup
Theorem \cite[Section 4.4]{Mar91}, $[\langle\Gamma^{j}\rangle,\langle\Gamma^{j}\rangle]<_{fi}\Gamma$
since it is an infinite normal group . Therefore $D<_{fi}\widehat{\Gamma}$.
As $\Gamma$ and $\Lambda$ are elements of $\cB_{0}$, their profinite
completions are naturally compact open subgroups of $G(\bA_{k,S})$.
We introduce some useful notation: For a prime $ $$p$ let
\begin{gather*}
X_{p}:=\{v:v\notin S,v|p\},\, H:=\prod_{v\in X_{p}}\Gamma_{v}\\
\fg:=Lie(H)=Lie_{\bQ_{p}}(\prod_{v\in X_{p}}G(k_{v})),\, H^{0}:=\prod_{v\in X_{p}}\Gamma_{v}^{0}.
\end{gather*}
Given a map $\Phi:A\ra B$ where $A$ and $B$ are two subsets of
$G(\bA_{k,S})$, let $\Phi_{\beta\alpha}:=\pi_{\beta}\circ\Phi\circ i_{\alpha}$
where$ $ $\alpha$ and $\beta$ are can be valuations of $k$ or
rational primes (the maps $i_{\alpha}$ and $\pi_{\beta}$ were defined
in subsection \ref{sub:Adeles}). 

We claim that for any rational prime $p$ and $w\in V^{k}\setminus S$
such that $w\nmid p$ we have that the image of $\Phi_{wp}:H\rightarrow G(k_{w})$
is a finite group. Indeed, as $H$ is a virtually pro-$p$ compact
group and so is its image. On the other hand, its image is also a
compact subgroup of $G(k_{w})$ and therefore it is virtually pro-$q$
where $w|q\neq p$ , hence is must be finite. 

Now we can use Jordan's Theorem to show that $\Phi_{wp}|_{H^{0}}:H^{0}\rightarrow G(k_{w})$
has trivial image: This is because $ $$\Phi_{wp}(H^{0})$ is generated
by 
\[
\{\Phi_{wv}(\overline{[\langle\Gamma_{v}^{j}\rangle,\langle\Gamma_{v}^{j}\rangle]})\}_{v\in X_{p}}
\]
and 
\begin{equation}
\Phi_{wv}(\overline{[\langle\Gamma_{v}^{j}\rangle,\langle\Gamma_{v}^{j}\rangle]})=\overline{[\langle(\Phi_{wv}(\Gamma_{v}))^{j}\rangle,\langle(\Phi_{wv}(\Gamma_{v}))^{j}\rangle]}=\{e\}\label{eq:jordan trivialization}
\end{equation}
$ $where last equality follows from Jordan's Theorem and since $\Phi_{wv}(\Gamma_{v})$
is a finite group by the last claim.

The above shows that $\Phi_{pp}|_{H^{0}}$ is injective and therefore
an isomorphism of the $p$-adic analytic group $H^{0}$ with its image
$\Phi_{pp}(H^{0})$. Thus the latter is a $p$-adic analytic compact
group of dimension $dim(H^{0})=dim(\fg)$. Thus $ $$\Phi_{pp}(H)$
is open in $ $$\prod_{v\in X_{p}}G(k_{v})$. Therefore, the derivative
of $\Phi_{pp}|_{H^{0}}$, which is also the derivative of $\Phi_{pp}|_{H}$,
is an automorphism of $\fg.$ 

Let $\fg=\oplus_{v\in X_{p}}\fg_{v}$ where $\fg_{v}=Lie(G(k_{v})).$
As $\fg$ is semisimple, there exist a permutation $\sigma\in Sym(X_{p})$
such that $d$ is induced by an isomorphism of $\fg_{v}$ with $\fg_{\sigma(v)}$,
$v\in X_{p}$. This imply that for any $w\neq\sigma(v)$ , $\Phi_{wv}:\Gamma_{v}\rightarrow G(k_{w})$
has finite image, and by the same argument as in (\ref{eq:jordan trivialization})
this also imply that $\Phi_{wv}|_{\Gamma_{v}^{0}}:\Gamma_{v}^{0}\rightarrow G(k_{w})$
has trivial image. Therefore, $\Phi_{\sigma(v)v}:\Gamma_{v}^{0}\rightarrow G(k_{\sigma(v)})$
is injective on $\Gamma_{v}^{0}.$

Let $\Delta_{w}:=\Phi_{\sigma(v)v}(\Gamma_{v}^{0})\subset G(k_{w})$
where $v$ and $\sigma$ are the unique valuation and permutation
that satisfy $\sigma(v)=w$ and supplied by the above argument. As
$d$ induce an isomorphism of $\fg_{v}$ with $\fg_{w}$, $\Delta_{w}$
is open in $G(k_{w})$. We now show that $\Lambda^{0}$ defined above
is of product-type by showing that $\widehat{\Lambda^{0}}=\prod_{w\notin S}\Delta_{w}$.
Since $\widehat{\Gamma^{0}}=\prod_{v\notin S}\Gamma_{v}^{0}$, the
group $\Phi(\widehat{\Gamma^{0}})=\widehat{\Lambda^{0}}$ is generated
by $\{\Phi\circ i_{v}(\Gamma_{v}^{0})\}_{v\notin S}.$ The above shows
that $\Phi\circ i_{v}(\Gamma_{v}^{0})$ has trivial $w$-coordinate
whenever $w\neq\sigma(v)$. Therefore, only the elements of $\Delta_{w}$
may appear as the $w$-coordinates of elements of $\widehat{\Lambda^{0}}$,
so $\widehat{\Lambda^{0}}\subseteq\prod_{w\notin S}\Delta_{w}$. Similarly,
any $(\alpha_{w})\in\prod_{w\notin S}\Delta_{w}$ is in $\widehat{\Lambda^{0}}$
since it is the image of $(i_{\sigma^{-1}(w)}^{-1}\circ\Phi^{-1}(\alpha_{w})_{w\notin S})\in\prod_{v\notin S}\Gamma_{v}^{0}=\widehat{\Gamma^{0}}$,
showing $\widehat{\Lambda^{0}}=\prod_{w\notin S}\Delta_{w}$. This
also shows that $\Phi|_{\widehat{\Gamma^{0}}}$ is a factor-type isomorphism
as it is induced from the isomorphisms $\Phi_{\sigma(v)v}|_{\Gamma_{v}^{0}}:\Gamma_{v}^{0}\rightarrow\Delta_{\sigma(v)}$.
\end{proof}

\subsection{\label{sub:proof within}Proof of Theorem \ref{thm: B family}}

In order to use Proposition \ref{prop:reduction to factor} we need
to find finite-index subgroups of general elements of $\cB$ that
will satisfy the assumptions of Proposition \ref{prop:reduction to factor},
so we begin by the following lemma:
\begin{lem}
\label{lem: finite index subgroups} Let $\Gamma\in\cB$. 
\begin{enumerate}
\item There exist $\Gamma^{0}<_{fi}\Gamma$ such that $\Gamma^{0}\subset G(k)$.
\label{enu:rational}
\item There exist $\Gamma^{0}<_{fi}\Gamma$ having the trivial group as
its congruence kernel which implies that $\ov{\Gamma^{0}}=\widehat{\Gamma^{0}}$,
i.e., the congruence completion of $\Gamma^{0}$ is equal to the profinite
completion of $\Gamma^{0}$. \label{enu:congruence}
\item Assume $\Gamma\in\cB_{0}$. There exist $\Gamma^{0}<_{fi}\Gamma$
which is of product-type. \label{enu:product type}\end{enumerate}
\begin{proof}
$ $$\quad$

As $\Gamma$ is commensurable to $G(\cO_{k,S})$, one can take $\Gamma^{0}=\Gamma\cap G(\cO_{k,S})$
for (\ref{enu:rational}). For (\ref{enu:congruence}), Note that
we are assuming that the congruence kernel of $G(\cO_{k,S})$ which
is denoted by $C(S,G)$ is finite. The congruence kernel of $\Gamma$,
$C(S,\Gamma)$ is contained in $C(S,G)$ hence it is also finite and
therefore discrete in the profinite topology on $\widehat{\Gamma}$.
Therefore, there exist a finite-index subgroup $H<\widehat{\Gamma}$
such that $C(S,\Gamma)\cap H=\{e\}$. Lemma \ref{lem:subgroups of profinite completion}
implies that $H$ is of the form $\widehat{\Gamma^{0}}$ with $\Gamma^{0}<_{fi}\Gamma$.
One can show that for any $\Delta<_{fi}\Gamma$, $C(S,\Delta)=C(S,\Gamma)\cap\widehat{\Delta}$,
thus the congruence kernel of $\Gamma^{0}$ is 
\[
C(S,\Gamma)\cap\widehat{\Gamma^{0}}=C(S,\Gamma)\cap H=\{e\}
\]
so we have $\ov{\Gamma^{0}}=\widehat{\Gamma^{0}}$.

Finally, for (\ref{enu:product type}), recall that the assumption
that $\Gamma\in B^{0}$ implies that $\widehat{\Gamma}=\overline{\Gamma}$
is an open compact subgroup of $G(\bA_{k,S})$. From the definition
of the topology on $G(\bA_{k,S})$, there exist a finite-index subgroup
$H$ of $\ov{\Gamma}\subset G(\bA_{k,S})$ of the form $H=\prod_{v\notin S}P_{v}$.
Again, by Lemma \ref{lem:subgroups of profinite completion}, $H$
is of the form $\widehat{\Gamma^{0}}$ with $\Gamma^{0}<_{fi}\Gamma$,
so $\Gamma^{0}$ is of product-type. 

\end{proof}
\end{lem}
\begin{proof}
(of theorem \ref{thm: B family}) Let $\Lambda$ be an arbitrary element
of $\cB_{\Gamma}$ and $\Phi:\widehat{\Gamma}\ra\widehat{\Lambda}$
be an isomorphism. Using Lemma \ref{lem: finite index subgroups}
we first find $\Lambda_{1}<_{fi}\Lambda$ such that $\Lambda_{1}\in\cB_{0}$.
Let $\Gamma_{1}<_{fi}\Gamma$ such that $\Phi_{*}(\Gamma_{1})=\Lambda_{1}$.
Then, again using Lemma \ref{lem: finite index subgroups}, we can
find $\Gamma_{2}<_{fi}\Gamma_{1}$ with the following properties:
$\Gamma_{2}\in\cB_{0}$ and it is of product-type. Let $\Lambda_{2}:=\Phi_{*}(\Gamma_{2})$
and note that $[\Gamma:\Gamma_{2}]=[\Lambda:\Lambda_{2}]$ and that
$\Lambda_{2}$ also belongs to $\cB_{0}$. 

Now, the groups $\Gamma_{2}$ and $\Lambda_{2}$ together with the
isomorphism $\Phi|_{\widehat{\Gamma_{2}}}:\widehat{\Gamma_{2}}\ra\widehat{\Lambda_{2}}$
satisfy the the conditions of Proposition \ref{prop:reduction to factor}.
Therefore, we find there exist $\Gamma_{3}<_{fi}\Gamma_{2}$ such
that $\Gamma_{3}$ and $\Lambda_{3}:=\Phi_{*}(\Gamma_{3})$ have a
factor-type isomorphism between them and thus the same covolume in
$G_{S}$ (Lemma \ref{lem:Covolume}). Moreover, 
\[
[\Gamma:\Gamma_{3}]=[\Lambda:\Phi_{*}(\Gamma_{3})]=[\Lambda:\Lambda_{3}],
\]
 so $\Gamma$ and $\Lambda$ also have the same covolume in $G_{S}$.
As $\Lambda$ was arbitrary, it follows that \emph{all} the elements
in $\cB_{\Gamma}$ has the same covolume. A theorem by Borel which
is a $S$-arithmetic extension a well-known theorem of Wang \cite{Borel87}
asserts that there are finitely many isomorphism classes of lattices
of bounded covolume in $G_{S}$, which is of $S$-rank $\geq2$. The
elements of $\cB$ are such lattices so this conclude the proof of
theorem \ref{thm: B family}.
\end{proof}

\section{Proof of the main Theorem \ref{thm:CSP}\label{sec:deduce main thm}}

By Theorem \ref{Main Theorem 1} there exist finitely many commensurability
classes, $\cC^{1},\dots,\cC^{r}$ and arithmetic groups $\Gamma_{1},\dots,\Gamma_{r}$
with $\Gamma_{i}\in\cC^{i}$ such that 
\[
\cA_{\Gamma}=\cup_{i=1}^{r}\cC_{\Gamma_{i}}^{i},
\]
where $\cC_{\Delta}^{i}:=\{\Lambda\in\cC^{i}:\widehat{\Lambda}\cong\widehat{\Delta}\}$.

Given an element $\Delta\in\cA$ a Theorem of Lubotzky \cite{Lub95}
characterize the property {}``$\Delta$ has the Congruence Subgroup
Property'' in terms of certain group-theoretic properties of $\widehat{\Delta}$.
Since $\Gamma$ has CSP and $\widehat{\Gamma}\cong\widehat{\Gamma_{i}}$
it follows that $\Gamma_{i}$ has CSP for all $i=1,\dots,r$. Finally,
for each $i$ there exist $(G_{i},k_{i},S_{i})$ such that any element
of $\cC^{i}$ has an isomorphic copy in $\cB_{G_{i},k_{i},S_{i}}$.
Therefore Theorem \ref{thm: B family} applied for $\cB_{G_{i},k_{i},S_{i}}$
implies that $\cC_{\Gamma_{i}}^{i}$ is a finite union of isomorphism
classes. As $i$ was arbitrary, it follows that $\cA_{\Gamma}$ is
also a finite union of isomorphism classes.

\bibliographystyle{amsalpha}
\bibliography{biblio}

\providecommand{\bysame}{\leavevmode\hbox to3em{\hrulefill}\thinspace}
\providecommand{\MR}{\relax\ifhmode\unskip\space\fi MR }
\providecommand{\MRhref}[2]{%
  \href{http://www.ams.org/mathscinet-getitem?mr=#1}{#2}
}
\providecommand{\href}[2]{#2}
\begin{thebibliography}{{Aka}11}

\bibitem[{Aka}11]{Aka10}
Menny {Aka}, \emph{{Profinite completions and Kazhdan's property (T)}}, To
  appear in Geometry, Groups and Dynamics (2011).

\bibitem[Bor87]{Borel87}
A.~Borel, \emph{On the set of discrete subgroups of bounded covolume in a
  semisimple group}, Proc. Indian Acad. Sci. Math. Sci. \textbf{97} (1987),
  no.~1-3, 45--52 (1988). \MR{MR983603 (90e:22013)}

\bibitem[Bou98]{bour98}
Nicolas Bourbaki, \emph{Lie groups and {L}ie algebras. {C}hapters 1--3},
  Elements of Mathematics (Berlin), Springer-Verlag, Berlin, 1998, Translated
  from the French, Reprint of the 1989 English translation. \MR{MR1728312
  (2001g:17006)}

\bibitem[BS64]{BS64}
A.~Borel and J.-P. Serre, \emph{Th\'eor\`emes de finitude en cohomologie
  galoisienne}, Comment. Math. Helv. \textbf{39} (1964), 111--164. \MR{0181643
  (31 \#5870)}

\bibitem[Dix71]{Dix71}
John~D. Dixon, \emph{The structure of linear groups}, Van Nostrand-Reinhold,
  London, New York,, 1971 (English).

\bibitem[GPS80]{GPS80}
F.~J. Grunewald, P.~F. Pickel, and D.~Segal, \emph{Polycyclic groups with
  isomorphic finite quotients}, Ann. of Math. (2) \textbf{111} (1980), no.~1,
  155--195. \MR{MR558400 (81i:20045)}

\bibitem[GZ11]{GZ11}
Fritz Grunewald and Pavel Zalesskii, \emph{Genus for groups}, J. Algebra
  \textbf{326} (2011), 130--168. \MR{2746056}

\bibitem[Hum75]{Hump75}
James~E. Humphreys, \emph{Linear algebraic groups}, Springer-Verlag, New York,
  1975, Graduate Texts in Mathematics, No. 21. \MR{MR0396773 (53 \#633)}

\bibitem[Kli98]{Klin98}
Norbert Klingen, \emph{Arithmetical similarities}, Oxford Mathematical
  Monographs, The Clarendon Press Oxford University Press, New York, 1998,
  Prime decomposition and finite group theory, Oxford Science Publications.
  \MR{MR1638821 (99k:11163)}

\bibitem[LS03]{LubotzkySegal}
Alexander Lubotzky and Dan Segal, \emph{Subgroup growth}, Progress in
  Mathematics, vol. 212, Birkh\"auser Verlag, Basel, 2003.

\bibitem[Lub95]{Lub95}
Alexander Lubotzky, \emph{Subgroup growth and congruence subgroups}, Invent.
  Math. \textbf{119} (1995), no.~2, 267--295. \MR{1312501 (95m:20054)}

\bibitem[Mar77]{Ma77}
Daniel~A. Marcus, \emph{Number fields}, Springer-Verlag, New York, 1977,
  Universitext. \MR{MR0457396 (56 \#15601)}

\bibitem[Mar91]{Mar91}
G.~A. Margulis, \emph{Discrete subgroups of semisimple {L}ie groups},
  Ergebnisse der Mathematik und ihrer Grenzgebiete (3) [Results in Mathematics
  and Related Areas (3)], vol.~17, Springer-Verlag, Berlin, 1991. \MR{MR1090825
  (92h:22021)}

\bibitem[Pin98]{Pink98}
Richard Pink, \emph{Compact subgroups of linear algebraic groups}, J. Algebra
  \textbf{206} (1998), no.~2, 438--504. \MR{MR1637068 (99g:20087)}

\bibitem[PR94]{PR94}
Vladimir Platonov and Andrei Rapinchuk, \emph{Algebraic groups and number
  theory}, Pure and Applied Mathematics, vol. 139, Academic Press Inc., Boston,
  MA, 1994, Translated from the 1991 Russian original by Rachel Rowen.
  \MR{MR1278263 (95b:11039)}

\bibitem[PR08]{PR08}
G.~{Prasad} and A.~S. {Rapinchuk}, \emph{{Developments on the congruence
  subgroup problem after the work of Bass, Milnor and Serre}}, ArXiv e-prints
  (2008).

\bibitem[Pra89]{Pr89}
Gopal Prasad, \emph{Volumes of {$S$}-arithmetic quotients of semi-simple
  groups}, Inst. Hautes \'Etudes Sci. Publ. Math. (1989), no.~69, 91--117, With
  an appendix by Moshe Jarden and the author. \MR{MR1019962 (91c:22023)}

\bibitem[RZ10]{RZ2010}
Luis Ribes and Pavel Zalesskii, \emph{Profinite groups}, second ed., Ergebnisse
  der Mathematik und ihrer Grenzgebiete. 3. Folge. A Series of Modern Surveys
  in Mathematics [Results in Mathematics and Related Areas. 3rd Series. A
  Series of Modern Surveys in Mathematics], vol.~40, Springer-Verlag, Berlin,
  2010. \MR{2599132 (2011a:20058)}

\bibitem[SW09]{SW09}
Y.~{Shalom} and G.~A. {Willis}, \emph{{Commensurated subgroups of arithmetic
  groups, totally disconnected groups and adelic rigidity}}, ArXiv e-prints
  (2009).

\end{thebibliography}

\end{document}